\documentclass[11pt, reqno,sumlimits]{article}
\usepackage{amssymb,amscd, epsfig, mathrsfs, xypic, amsmath,amsthm, color} 
\usepackage{hyperref}
\usepackage{blindtext}

\xyoption{all}
\newtheorem{thm}{Theorem}[section] 

\newtheorem{lem}[thm]{Lemma}
\newtheorem{prop}[thm]{Proposition} 
\newtheorem{df-pr}[thm]{Definition-Proposition}

\theoremstyle{definition}
\newtheorem{defn}[thm]{Definition}
 
\newtheorem{rem}[thm]{Remark}

\newtheorem{exm}[thm]{Example}

\newcommand{\ZZ}{{\mathbb Z}}

\newcommand{\PP}{{\mathbb P}}
\newcommand{\LL}{{\mathbb L}}
\newcommand{\NN}{{\mathbb N}}

\newcommand{\sfm }{{\mathsf m }}

\newcommand{\sfs }{{\mathsf s}}

\newcommand{\frakX}{{\mathfrak X }}

\newcommand{\calA}{{\mathcal A}}

\newcommand{\calC}{{\mathcal C}}

\newcommand{\calO}{{\mathcal O}}
\newcommand{\calP}{{\mathcal P}}
\newcommand{\calQ}{{\mathcal Q}}

\newcommand{\calS}{{\mathcal S}}

\newcommand{\scA}{{\mathscr A}}

\newcommand{\scC}{{\mathscr C}}

\newcommand{\scL}{{\mathscr L}}

\newcommand{\scS}{{\mathscr S}}

\newcommand{\Fl}{\operatorname{Fl}}

\newcommand{\Gr}{\operatorname{Gr}}

\newcommand{\gr}{\operatorname{gr}}

\newcommand{\rk}{{\operatorname{rk}}}

\newcommand{\supp}{{\operatorname{supp}}}

\newcommand{\Det}{{\operatorname{Det}}}
\newcommand{\Pf}{{\operatorname{Pf}}}
\newcommand{\LG}{{\operatorname{LG}}}
\newcommand{\SP}{{{\calS\calP}}}

\newcommand{\SM}{\mathbf{Sm}_k}

\newcommand{\Laz}{\mathbb{L}}
\newcommand{\trecd}{\cdot\cdot\cdot}
\newcommand{\tred}{\ldots}

\newcommand{\spec}{{\rm Spec\,}}
\newcommand{\SE}{\mathcal{S}}

\newsavebox{\savepar}

\numberwithin{equation}{section}

\newcounter{labelflag} 
\newcommand{\labelon}{\setcounter{labelflag}{1}}
\newcommand{\Label}[1]{\ifnum\thelabelflag=1\ifmmode
\makebox[0in][l]{\qquad\fbox{\rm#1}} \else
\marginpar{\vspace{0.7\baselineskip} \hspace{-1.1\textwidth}
\fbox{\rm#1}} \fi \fi \label{#1} } \labelon

\begin{document} 
\title{Kempf-Laksov Schubert classes for even infinitesimal cohomology theories}
\author{Thomas Hudson, and Tomoo Matsumura}
\date{}
\maketitle 
\begin{abstract}
In this paper, we prove a generalization of Kempf-Laksov formula for the degeneracy loci classes in even infinitesimal cohomology theories of the Grassmannian bundle and the Lagrangian Grassmannian bundle. 
\end{abstract}

\section{Introduction}
This work was motivated by the attempt to generalize the Kempf-Laksov formula of the Chow ring to general oriented cohomologies. In 1902, Giambelli \cite{Giambelli} described the fundamental classes of the Schubert varieties of the Grassmannian in a closed, determinantal expression involving the Chern class of the tautological vector bundle. Later, in 1974, Kempf and Laksov \cite{KempfLaksov} generalized Giambelli's formula to the Grassmann bundles associated to a vector bundle. The proof essentially consists in the construction, through a tower of projective bundles, of a resolution of singularities and in a Gysin computation which produces the formula for the fundamental class. Recently,  in \cite{HIMN},  together with T. Ikeda and H. Naruse, we generalized such computation to $K$-theory. 
  In this context we managed to obtain a determinantal formula for the Schubert classes. This was achieved by combining the geometric input given by Kempf-Laksov's resolution, with an algorithmic procedure modelled after the one used by Kazarian in \cite{Kazarian}. Given this state of things it looked reasonable to try and see whether or not this procedure could be further generalized to more general oriented cohomology theories. 

The aim of this paper is to prove a generalization of Kempf-Laksov formula in the infinitesimal theories $I_n^*$, the easiest examples of oriented cohomology theories beyond $K$-theory. Due to the fact that not all Schubert varieties have a well defined notion of fundamental class in a general oriented cohomology theories, the computation of the Schubert classes depends on the choice of the resolution, which in general is not unique. In our context, it is natural to consider the classes associated to Kempf-Laksov's resolutions. The advantage of considering those classes is that they are stable along the natural inclusions connecting the Grassmann bundles associated to bundles of increasing rank. This allows us to define a generalization of Schur/Grothendieck polynomials in the context of more general oriented cohomology theories. It is worth pointing out that for flag varieties and flag bundles another generalization of Schubert classes was studied in \cite{SchubertCalmes,EquivariantCalmes,SchubertHornbostel,ThomHudson,GeneralisedHudson,EquivariantKiritchenko}, by taking Bott-Samelson resolutions into consideration.


The notion of oriented cohomologies in algebraic geometry was introduced by Levine-Morel in \cite{LevineMorel}, who were inspired by the work \cite{Quillen} of Quillen in the category of differential manifolds. Among such theories, algebraic cobordism $\Omega^*$ is the universal one and it can be used to construct other theories with prescribed formal group laws: the infinitesimal theory $I_n^*$ is one of such examples. An oriented cohomology theory is characterized by the formal group law which encodes the expansion of the first Chern class of the tensor product of two line bundles. In particular, as one of the features of the universal cohomology theory, Levine-Morel identified the coefficient ring of $\Omega^*$ with the Lazard ring $\LL$ equipped with the universal formal group law.

Let $\alpha_n$ be an indeterminant, and let $Q_n:=\ZZ[\alpha_n]$ with $\alpha_n^2=0$ and $\deg \alpha_n = -n$. In Section \ref{subsect Lazard}, we introduce the surjective ring homomorphism $\LL \to Q_n$, following \cite[Part II, \S 5 and \S 7]{StableAdams}. We then define the infinitesimal cohomology theory $I_n^*$ by $I_n^*:= \Omega^*\otimes_{\LL} Q_n$. In this paper, we work with the case $n=2m$ for simplicity. Then the formal group law for $I_{2m}^*$ is given by (Lemma \ref{lemFGL2m})
\[
u\boxplus v=(u+v)\left[ 1+\alpha_{2m}\left( \sum_{i=1}^{2m-1}\gamma^{(2m)}_i u^iv^{2m-i}\right) \right].
\]
Here one sets
\[
\gamma^{(2m)}_i=\frac{1}{d_{2m}}\left[\binom{2m}{i} -(-1)^i\right],
\]
where $d_i=p$ if $i=p^e-1$ for some integer $e$ and a prime $p$, and otherwise let $d_i=1$. It is clear from the expression of the formal group law that the formal inverse is given by $\boxminus u=-u$.

A key role in our computation is played by the Segre classes $\scS_m(E)$ of vector bundles $E$, as it was in \cite{HIMN}. The definition we use is the exact analogue of that given by Fulton in \cite{FultonIntersection}. Using Quillen's formula, which for $\Omega^*$ was established by Vishik in \cite{Vishik}, we manage to describe the generating function of Segre classes in terms of Chern classes. This allows us to define the relative Segre classes $\scS_m(E-F)$ evaluated on each element $[E-F]$ of the Grothendieck group of vector bundles. Then, the following geometric interpretation is the main ingredient of our derivation of Kempf-Laksov formula.

\vspace{2mm}

\noindent{\bf Proposition A} (Proposition \ref{push of tensor}).
\emph{
Let $E$ and $F$ vector bundles over a smooth scheme $X$, respectively of rank $e$ and $f$. Consider the dual projective bundle $\PP^*(E)\stackrel{\pi}\rightarrow X$ with the tautological quotient line bundle $\calO(1)$. Then for each integer $s\geq 0$ we have
\[
\pi_*\Big(c_1(\calO(1))^s c_f(\calO(1)\otimes  F^{\vee})\Big)=\SE_{s+f-e+1}(E- F).
\]
}
We now explain our main result. Let $E$ be a vector bundle of rank $n$ over a smooth quasi-projective variety $X$. Consider  the Grassmann bundle of rank $d$ subbundles $\Gr_d(E) \to X$ and let $S$ be its the tautological vector bundle. Fix a complete flag $0=F^n\subset \cdots \subset F^1\subset E$ where the superscript indicates the corank. Let $\lambda=(\lambda_1,\dots,\lambda_r)$ be a partition of length $r$ such that $\lambda_1\leq n-d$. Let $X_{\lambda} \subset \Gr_d(E)$ be the associated degeneracy locus and $\widetilde{X}^{KL}_\lambda \to X_{\lambda}$ be its Kempf-Laksov resolution (see Section \ref{secKL}). Our explicit closed formula of the class $\kappa_{\lambda}$ of the resolution $\widetilde{X}^{KL}_\lambda \to \Gr_d(E)$ in $I_{2m}^*(\Gr_d(E))$ is as follows. Let $\scA_s^{(\ell)} := \scS_s(S^{\vee}- (E/F^{\ell})^{\vee})$ for $s\in \ZZ$ and $\ell=1,\dots,n$. Recall that the multi-Schur determinant is denoted as follows. For each $(s_1,\dots, s_r) \in \ZZ^r$ and $k_1,\dots, k_r\in \{0,1,\dots, n\}$, denote
\[
\Det[\calA^{(k_1)}_{s_1}\cdots \calA^{(k_r)}_{s_r}]=\det(\calA_{s_i+j-i}^{(k_i)})_{1\leq i,j\leq r}.
\]
\noindent{\bf Theorem B} (Theorem \ref{mainthmA}). \emph{
Let $\lambda$ be a partition in $\calP_d(n)$ with length $r$. Let $k_i:=\lambda_i-i+d$. The Kempf-Laksov class $\kappa_{\lambda}$ in $I_{2m}^*(\Gr_d(E))$ is given by
\begin{eqnarray*}
\kappa_{\lambda} &=& \Det\big[\calA^{(k_1)}_{\lambda_1}\cdots \calA^{(k_r)}_{\lambda_r}\big]\\
&&\ \ \ \ \ +
\alpha_{2m} \sum_{l=-m+1}^{m-1} (-1)^{m+l}\gamma_{m+l}^{(2m)} \left( \sum_{1\leq a<b \leq r} 
\Det\big[\calA^{(k_1)}_{\lambda_1}\cdots \calA^{(k_a)}_{\lambda_a+m+l}\cdots \calA^{(k_b)}_{\lambda_b+m-l}  \cdots \calA^{(k_r)}_{\lambda_r}\big]
\right).
\end{eqnarray*}
}

In Section \ref{secLagC}, we also apply our method to the computation of the Kempf-Laksov classes for the Lagrangian Grassmann bundles  and obtain a Pfaffian formula. We believe that it can be applied to a more general setting, which should include the orthogonal Grassmann bundles as well. In fact, in \cite{AndersonFulton2}, Anderson-Fulton  applied Kazarian's method to the degeneracy loci defined by vexillary (signed) permutations. We will pursue the general formulas describing those degeneracy loci in our future works.
\section{Algebraic cobordism and infinitesimal theories}
The main goal of this section is to introduce infinitesimal theories and explain how they relate to algebraic cobordism. This will require us to recall some basic facts about oriented cohomology theories, formal group laws and the construction of the Lazard ring.
\subsection{Oriented cohomology theories}
 In algebraic geometry the notion of oriented cohomology theory is due to Levine and Morel, who introduced it in \cite{LevineMorel}, inspired by the work of Quillen on differentiable manifolds (\cite{Quillen}). Such a theory consists of a contravariant functor  from the category of smooth schemes to graded abelian rings $A^*:\SM^{op}\rightarrow \mathcal{R}^*$, together with a family of push-forward maps $\{f_*:A^*(X)\rightarrow A^*(Y)\}$ associated to the projective morphisms $\{X\stackrel{f}\rightarrow Y\}$. This family is supposed to satisfy some straightforward functorial properties and to be compatible with the pull-back maps $g^*$ arising  from $A^*$, whenever $f$ and $g$ are transverse. Finally, $A^*$ is supposed to satisfy the extended homotopy property and the projective bundle formula, which respectively relate the evaluation of $A^*$ on vector and projective bundles to the evaluation on their bases. For the precise definition we refer the reader to \cite[Definition 1.1.2]{LevineMorel}.

Fundamental examples of oriented cohomology theories are the Chow ring $CH^*$ and $K^0[\beta,\beta^{-1}]$, a graded version of the Grothendieck ring of vector bundles $K^0$ obtained by tensoring with the ring of Laurent polynomials $\ZZ[\beta,\beta^{-1}]$, with $\text{deg\,}\beta=-1$. Many of the general features of oriented cohomology theories can be seen already in these two examples. For instance, it is possible to define a theory of Chern classes $c_i^A$, which shares many of the properties of the one given in the Chow ring.

 In order to understand in which ways the concept of oriented cohomology theory is more general, it can be useful to pay a closer look at the behaviour of the first Chern class of line bundles. It is well known that $c_1^{CH}$ behaves linearly with respect to tensor product: for line bundles $L$ and $M$ over $X\in \SM$ one has 
\[
c^{CH}_1(L\otimes M)=c_1^{CH}(L)+c_1^{CH}(M).
\]
Although this equality does not necessarily hold if we replace the Chow ring with another oriented cohomology theory $A^*$, it is still possible to express $c_1^A(L\otimes M)$ in terms of the first Chern classes of the factors. However, for this one has to replace the usual sum with a formal group law $F_A$.

Before we continue with our discussion, let us now briefly recall the definition of a formal group law. A pair $(R,F_R)$, where $R$ is a ring and $F_R(u,v)\in R[[u,v]]$, is said to be a commutative formal group law of rank 1 if it satisfies the following conditions:
\begin{align}\label{eq linear}
i)\ F_R(u,0)&=F_R(0,u)=u\in R[[u]];
\\ii)\ F_R(u,v)&=F_R(v,u)\in R[[u,v]]; \label{eq comm}
\\iii)\ F_R(u,F_R(v,w))&=F_R(F_R(u,v),w)\in R[[u,v,w]]. \label{eq assoc}
\end{align}     
It should be noticed that $ii)$ and $iii)$ can be viewed as analogues of the commutative and associative properties for groups. Actually, this analogy can be pushed further since there exists also a notion of formal inverse, a power series $\chi(u)\in R[[u]]$ such that $F_R(u, \chi(u))=0$. Let us finish this digression by mentioning that there exists a universal formal group law $F$, defined over the Lazard ring $\Laz$, from which all other ones can be derived. We will return to this point in Section \ref{subsect Lazard}.

As we hinted at before, every cohomology theory $A^*$ has an associated formal group law $F_A$, defined over its coefficient ring $A^*(\spec k)$, such that in $A^1(X)$ one has
\[
c_1^A(L\otimes M)=F_A(c_1^A(L),c_1^A(M))
\]
for any choice of line bundles $L$ and $M$ over the given scheme $X$. As we have seen $F_{CH}(u,v)=u+v$, while for $K^0[\beta,\beta^{-1}]$ one has $F_{K^0[\beta,\beta^{-1}]}=u+v-\beta\cdot uv$.

The main achievement of Levine and Morel concerning oriented cohomology theories consists in the construction of algebraic cobordism, denoted $\Omega^*$, which they identify as universal in the following sense.

\begin{thm}[\protect{\cite[Theorems 1.2.6 and 1.2.7]{LevineMorel}}]
Let $k$ be a field of characteristic 0. $\Omega^*$ is universal among oriented cohomology theories on $\SM$. That is, for any other oriented cohomology theory $A^*$ there exists a unique morphism 
$$\vartheta_A:\Omega^*\rightarrow A^*$$
of oriented cohomology theories. Furthermore, its associated formal group law $(\Omega^*(\spec k),F_\Omega)$ is isomorphic to $(\Laz,F)$, the universal one.
\end{thm}

For us a very important feature of $\Omega^*$ is that it allows to construct new oriented cohomology theories. In fact, for every $(R,F_R)$ it is sufficient to consider
$$\Omega^*_{(R,F_R)}:=\Omega^*\otimes_\Laz R$$ 
to obtain a theory whose associated formal group law is the given one and, essentially by design, it is universal among such theories. For instance, the formal group law $u\oplus v=u+v-\beta\cdot uv$ defined over $\ZZ[\beta]$ gives rise to connective $K$-theory, denoted $CK^*$, which has been considered in \cite{ThomHudson,HIMN}.
\subsection{Chern classes}
The existence of a theory of Chern classes for an oriented cohomology theory $A^*$ is a direct consequence of the projective bundle formula. This claims that, for any given bundle $E\rightarrow X$ of rank $n$, one has the following isomorphism of $A^*(X)$-modules: 
$$A^*(\PP (E))\simeq \bigoplus_{i=0}^{n-1}\xi^iA^*(X).$$   
 Here $\xi=c_1(\calO(1))$ and $\calO(1)$ stands for the universal quotient line bundle over $\PP^*(E)$. Up to a sign Chern classes should be viewed as the coefficients of the expansion of $\xi^n$ with respect to the basis $\{1,\xi,...,\xi^{n-1}\}$. More precisely, in $A^n(X)$ one has the following defining equality
\begin{align}\label{eq Chern} 
\xi^n-c_1(E)\xi^{n-1}+\trecd+(-1)^n c_n(E)=0.
\end{align}
It can be convenient to assemble together the Chern classes in the Chern polynomial $c_t(E):=\sum_{i=0}^n c_i(E)t^i$. In fact, one has the following proposition, usually known as the Whitney product formula.
\begin{prop}[\protect{\cite[Proposition 4.1.15 (3)]{LevineMorel}}]
Let 
$$0\rightarrow E\rightarrow F\rightarrow G\rightarrow 0$$
 be a short exact sequence of vector bundles over $X\in \SM$. Then in $A^*(X)[t]$ one has 
$$c_t(F)=c_t(E)\cdot c_t(G).$$
\end{prop} 
It is important to notice that the Whitney formula guarantees that the assignment $E\mapsto c_t(E)$ is actually well defined at the level of the Grothendieck group $K^0(X)$ and, as a consequence, one can associate a Chern polynomial to any of its element. In particular for $[E]-[F]$ one obtains
\begin{align} \label{eq virtual}
c_t(E-F)=\frac{c_t(E)}{c_t(F)}.
\end{align}
\subsection{The construction of the Lazard ring and its multiplicative structure} \label{subsect Lazard}
In order to be able to express the formal group laws that give rise to the infinitesimal theories, it is convenient to recall some known facts about the structure of $\Laz$ and its definition. We will essentially follow the treatment given in \cite[Part II, \S 5 and \S 7]{StableAdams}, trying to make it compatible with the notations used in \cite[Chapter 1]{LevineMorel}.
 
Let us consider the graded polynomial ring $\ZZ[A_{i,j}]_{(i,j)\in\NN^2}$ with $\text{deg}\,A_{i,j}=-i-j+1$ and the power series $$\widetilde{F}(u,v)=\sum_{i,j}A_{i,j}u^i v^j.$$ The Lazard ring can be defined as the quotient  obtained by forcing $\widetilde{F}$ to be a formal group law and for this it is sufficient to impose conditions (\ref{eq linear}),(\ref{eq comm}), and (\ref{eq assoc}). If we set $a_{i,j}=\overline{A_{i,j}}$, then the formal group law of $\Laz$ can be written as
\[
F(u,v)=u+v+\sum_{i,j\neq 0}a_{i,j}u^i v^j.
\]
An obvious consequence of this construction is that $(\Laz,F)$ represents the universal formal group law: for any other pair $(R,F_R)$ there exists a unique morphism $\varPhi_{(R,F_R)}:\Laz\rightarrow R$ mapping the coefficients $a_{i,j}$ to those of $F_R$. From now on we will write $\varPhi_R$ instead of the more precise $\varPhi_{(R,F_R)}$. It is possible to make use of this universality to investigate the structure of $\Laz$ by embedding it into a polynomial ring. For this one sets $R=\ZZ[b_1,b_2,\tred,b_i,\tred]$, with $\text{deg}\,b_i=-i$ and considers the two power series 
\[
exp(u)=\sum_{i=0}^\infty b_i u^i \qquad \text{and}\qquad
log(u)=\sum_{i=0}^\infty m_i u^i,
\]
where $b_0=1$ and, as one requires $log$ to be the inverse of $exp$, the coefficients $m_i$'s are uniquely determined. Finally, with these two power series at hand it is possible to define a formal group law by setting 
\[
F_R(u,v)=exp\Big(log(u)+log(v)\Big)
\]
and obtain $\varPhi_R$. A priori it is not clear that this map is injective, however it is possible to draw this and other conclusions after studying the structure of $\Laz$ and $R$ modulo decomposable elements.
If $J_\Laz$ denotes the ideal of $\Laz$ which consists of all the elements of negative degree, we then refer to $J_\Laz^2$ as the decomposable elements and, in the same fashion, we can also consider $J_R^2$. Although the coefficients of $F_R$ are not straightforward to express, the situation simplify considerably if one looks at $F_{R/J_R^2}$. In fact, one has the following proposition.
\begin{prop}
\begin{enumerate}
\item[1)] $\varPhi_{R/J_R^2}(a_{i,j})=\binom{i+j}{i} \overline{b_{i+j-1}}.$
\item[2)] As a group, the image of $\varPhi_{R/J_R^2}$ can be described as follows: 
\[
\text{Im}\,\varPhi_{R/J_R^2}=\bigoplus_{i=0}^\infty\   d_i \overline{ b_i}\cdot\ZZ,
\]
where one sets $d_i=p$ if $i=p^e-1$ for some integer $e$ and a prime $p$, and otherwise let $d_i=1$.
\item[3)] The induced morphism $\overline{\varPhi_{R_/J_R^2}}:\Laz/J_\Laz^2\rightarrow R/J_R^2$ is injective.  
\end{enumerate}
\end{prop}
\begin{proof}
1) and 2) are restatements  of \cite[Lemma 7.9 ii) and iii)]{StableAdams}, while 3) is equivalent to \cite[Corollary 14]{StableAdams} provided one considers separately the different graded pieces.
\end{proof}
Let us observe that, in view of the third part of the proposition, we can identify $Q:=\Laz/J_\Laz^2$ with the image described in the second part and, moreover, the first statement provides us with an explicit description of the formal group law associated to the projection $\Laz\rightarrow Q$. In fact, it is equivalent to saying that
\begin{eqnarray}\label{eq FGL}
F_Q(u,v)=u+v+\sum_{i,j\neq 0}\binom{i+j}{i} \frac{1}{d_i}\mathfrak{\alpha}_{i+j-1}\cdot u^iv^j,
\end{eqnarray}
where $\alpha_i\in Q^{-i}$ is the unique element which maps to $d_i\overline{b_i}$. It is worth noticing that, from a multiplicative point of view, the different graded components of $Q^{-i}$ are independent of each other, so it is possible to study them separately. In particular, since we are interested in the graded components of the formal group law, we will consider the rings $Q_i:=Q^0\oplus Q^{-i}$.   

We complete our discussion on the Lazard ring by presenting another consequence of the previous proposition. In fact the study of $Q$ and of $R/J_R^2$ allows one to give a precise description of the multiplicative structure of $\Laz$.
\begin{thm}
\begin{itemize}
\item[(1)]The Lazard ring is isomorphich to the polynomial ring $\ZZ[x_1,x_2,\tred,x_i,\tred]$ with $\text{deg}\, x_i=-i$, where $x_i$ is mapped to any element of the preimage $\varPhi^{-1}_{R/J_R^2}(d_i\overline{b_i})$.
\item[(2)] The map $\varPhi_R$ is a monomorphism.
\end{itemize}
\end{thm} 
\begin{proof}
See \cite[Theorem 7.1 and 7.8]{StableAdams}.
\end{proof}
\subsection{Infinitesimal theories}
We are now ready to introduce the oriented cohomology theories that will be our object of study.
\begin{defn}
Let $n\in\NN$ be a strictly positive integer. The $n$-th infinitesimal oriented cohomology theory $I^*_n$ is defined as
\[
I^*_n:= \Omega^*\otimes_\Laz Q_n,
\]
where the tensor product is taken with respect to the classifying morphism $\varPhi_{Q_n}$.
In view of (\ref{eq FGL}) the associated formal group law, denoted $\boxplus$, is given by
\[
u\boxplus v:= F_{Q_n}(u,v)=u+v+\alpha_n\cdot \frac{1}{d_n}\sum_{i=1}^{n} \binom{n+1}{i}  u^iv^{n+1-i}.
\]
\end{defn}
The following lemma shows that, in the case of even infinitesimal theories, the formal group law has a special shape. Note that the formal inverse $\chi_{I_n}$ is denoted by $\boxminus$. 
\begin{lem}\label{lemFGL2m}
Suppose that $n=2m$ is even. Then 
\begin{align}\label{eq gamma}
u\boxplus v=(u+v)\left[ 1+\alpha_{2m}\left( \sum_{i=1}^{2m-1}\gamma^{(2m)}_i u^iv^{2m-i}\right) \right],
\  \text{with} \quad \gamma^{(2m)}_i=\frac{1}{d_{2m}}\left[\binom{2m}{i} -(-1)^i\right].
\end{align}
Moreover, one has $\boxminus u=-u$.
\end{lem}
\begin{proof}
It is easy to see that the coefficient of $\alpha_{2m}\frac{1}{d_{2m}}$ is equal to $(u+v)^{2m+1}-u^{2m+1}-v^{2m+1}$ and that one has the following equality 
$$u^{2m+1}+v^{2m+1}=(u+v)[u^{2m}-u^{2m-1}v+\trecd+(-1)^i u^{2m-i}v^i+\trecd-uv^{2m-1}+v^{2m}].$$
It is then sufficient to factor $(u+v)$ and use the binomial expasion to obtain the desired formula for $\gamma_i^{(2m)}$. The statement concerning the formal inverse, since it is uniquely defined, immediately follows from the factorization.
\end{proof}
\begin{exm}
For $n=2$ and $n=4$ the formal group laws are respectively given by: 
\begin{align*}
i)\ F_{Q_2}(u,v)=u+v+\alpha_2\Big(uv^2&+u^2 v\Big)=(u+v)\Big[1+\alpha_2\cdot uv\Big];\\
ii)\ F_{Q_4}(u,v)=u+v+\alpha_4\Big(uv^4+2\cdot u^2v^3
&+2\cdot u^3v^2+u^4v\Big)=(u+v)\Big[1+\alpha_4\Big(uv^3+u^2v^2+u^3 v\Big)\Big].
\end{align*}
\end{exm}
\section{Segre classes for infinitesimal theories}
The main aim of this section is to introduce the notion of Segre classes for $I_{2m}^*$ and to provide explicit descriptions in terms of symmetric functions. 

Let $e_i(x_1,\dots, x_e)$ and $h_i(x_1,\dots, x_e)$ be the $i$-th elementary and complete symmetric polynomials of variables $x_1,\dots, x_e$. If we regard $x_1,\dots, x_e$ as Chern roots of a vector bundle $E$, we write $e_i(E)$ and $h_i(E)$ instead. We also use the \emph{power sum symmetric polynomials}
\[
p_i(x_1,\dots, x_e) := \sum_{k=1}^e x_k^i
\]
for each integer $i\geq 1$. Again, if we regard $x_i$'s as Chern roots of a vector bundle $E$, we write $p_i(E)=p_i(x_1,\dots, x_e)$. In this paper, we use the convention $p_0(x_1,\dots, x_e)=1$.

We now give a geometric definition of Segre classes and describe them explicitly through the use of symmetric functions. 
\begin{defn}
Let $E\rightarrow X$ be a vector bundle of rank $e$ and $\PP^*(E)\stackrel{\pi}\rightarrow X$ the associated projective bundle of hyperplanes. Let $\tau:= c_1(\calO(1))\in I^1_{2m}(\PP^*(E))$. For each $k \geq -e+1$, we define the $k$-th (generalized) Segre classes by setting 
\[
\SE_k(E):=\pi_*(\tau^{k+e-1}).
\]
\end{defn}
\begin{prop}\label{prop Segre}
In $I_{2m}^*(X)$, the Segre classes of $E$ are given by:
\begin{align}\label{eq Segre}
\mathcal{S}_k(E)=h_{k}(E)-\alpha_{2m}\sum_{l=0}^{2m-1} (-1)^l \gamma_l^{(2m)}\cdot  p_l(E) \cdot h_{2m+k-l}(E),
\end{align}
where $\gamma_0^{(2m)}:=-\sum_{l=1}^{2m-1}(-1)^l \gamma_l^{(2d)}=\frac{2m+1}{d_{2m}}$.
\end{prop}
\begin{proof}
Before we begin the proof of the main statement, let us quickly observe that, in view of (\ref{eq gamma}), the closed formula for $\gamma_0^{(2m)}$ easily follow from its definition. 
If the $x_i$'s represent the Chern roots of $E$, then starting from push-forward formula of Vishik (\cite[Proposition 5.29]{Vishik}), we have following chain of equalities:
\begin{eqnarray*}
\mathcal{S}_{k}(E)
&=&\sum_{i=1}^e \frac{x_i^{k+e-1}}{ \prod_{j\neq i}\left[(x_i- x_j)\left[1+\alpha_{2m}\left(\sum_{l=1}^{2m-1} \gamma_l^{(2m)} x_i^{2m-l}(-x_j)^l\right)\right]\right]}\\
&=&\sum_{i=1}^e \frac{x_i^{k+e-1}\prod_{j\neq i}\left[1-\alpha_{2m}\left(\sum_{l=1}^{2m-1} \gamma_l^{(2m)} x_i^{2m-l}(-x_j)^l\right) \right]}{\prod_{j\neq i}(x_i- x_j)}\\
&=&\sum_{i=1}^e\frac{x_i^{k+e-1}\left[1-\alpha_{2m}\sum_{l=1}^{2m-1}\left((-1)^l\gamma_l^{(2m)}x_i^{2m-l}(p_l(E)-x_i^l)\right)\right]}{\prod_{j\neq i}(x_i- x_j)}.
\end{eqnarray*}
It is well-known that the complete symmetric polynomials has the following expression:
\[
\sum_{i=1}^e \frac{x_i^k}{\prod_{j=1,\dots, e \atop{j\not= i}} (x_i-x_j)}=h_{k-e+1}(x_1,\tred,x_e).
\]
Therefore
\begin{eqnarray*}
\mathcal{S}_{k}(E)
&=&h_k(E)-\alpha_{2m}\left(\sum_{l=1}^{2m-1}(-1)^l \gamma_l^{(2m)}p_l(E) \cdot h_{2m+k-l}(E)-\sum_{l=1}^{2m-1}(-1)^l \gamma_l^{(2m)}\cdot h_{2m+k}(E)\right).
\end{eqnarray*}
Thus the claim follows.
\end{proof}
\begin{rem}
A direct consequence of the previous proposition is that $S_k(E)$ does not change if the vector bundle $E$ is replaced by $E\oplus\mathcal{O}_X^m$, for any positive $m$. Hence, it is possible to extend the definition of $S_k$ to arbitrary indices, bypassing the restriction given by the rank. In fact we can set 
\[
\SE_k(E):=\lim_{m\rightarrow \infty} \SE_k(E\oplus\mathcal{O}_X^m).
\]
Thus $\SE_k(E)$ can be defined for all $k\in \ZZ$. 
\end{rem}
Consider the generating function of Segre classes of $E$:
\[
\SE_t(E):=\sum_{k\in\ZZ} \SE_k(E) t^k.
\] 
Our next goal is to provide a description of $\SE_t(E)$. 
\begin{prop}\label{prop R}
Let $E$ be a vector bundle of rank $e$ over $X\in\SM$. Then in $I_{2m}^*(X)[[t]]$ we have 
\[
\SE_t(E)=c_{-t}(-E)\left(1-\alpha_{2m}\sum_{i=1}^{2m} (-1)^i \gamma^{(2m)}_{2m-i}p_{2m-i}(E)t^{-i}\right).
\]  
\end{prop}
\begin{proof}
To prove the formula, we compute $R_t(E):=\SE_t(E)c_{-t}(E)$. By definition we have 
\begin{equation}\label{eqR_i}
R_i(E)=\sum_{q=0}^e (-1)^q c_q(E)\SE_{i-q}(E).
\end{equation}
We begin by proving that $R_i=0$ for $i>0$. This is an direct  consequence of (\ref{eq Chern}). In fact, by multiplying both sides of (\ref{eq Chern}) by $\xi^{i-1}$ and then  pushing forward to $I_{2m}^*(X)$, one obtains $R_i(E)=0$. On the other hand, if $i<-2m$, then $R_i$ is again zero since by Proposition \ref{prop Segre} all the Segre classes involved vanish. Let us now assume that $-2m \leq i\leq 0$. By applying Proposition \ref{prop Segre} to (\ref{eqR_i}), one obtains
\begin{eqnarray*}
R_i(E)
&=&\sum_{q=0}^e (-1)^q c_q(E)h_{i-q}(E)-\alpha_{2m}\sum_{l=0}^{2m-1} (-1)^l \gamma_l^{(2m)}\cdot  p_l(E) \cdot \sum_{q=0}^e (-1)^q c_q(E)h_{2m+i-q-l}(E).
\end{eqnarray*}
The claim follows by observing that the given alternating sums of the product of complete and elementary symmetric functions are $1$ if the total degree is $0$ and vanishes otherwise.
\end{proof}
\begin{defn}
We define the relative Segre classes $\SE_i(E-F)$ by evaluating on $[E]-[F]$ the symmetric functions in (\ref{eq Segre}). Or equivalently, define by Proposition \ref{prop R}:
\begin{equation}\label{dfRelSeg}
\SE_t(E-F)=c_{-t}(-E+F)\left(1-\alpha_{2m}\sum_{i=1}^{2m} (-1)^i \gamma^{(2m)}_{2m-i}p_{2m-i}(E-F)t^{-i}\right).
\end{equation}
\end{defn}
Next we want to express $S_t(E-F)$ in terms of symmetric functions and the Segre and Chern polynomials of the two summands.
\begin{prop}\label{relSegre}
Let $E$ and $F$ be two vector bundles. Then the Segre polynomial of the associated virtual class is given by
\[
\SE_t(E-F)=c_t(F^\vee)\SE_t(E)\left[ 1+ \alpha_{2m} \sum_{i=1}^{2m-1}\gamma_{2m-i}^{(2m)} p_{2m-i}(F^\vee)t^{-i}\right].
\]
\end{prop}
\begin{proof}
In (\ref{dfRelSeg}), Lemma \ref{lemma Power} allows one to rewrite $p_k(E-F)$, and by applying Proposition \ref{prop R} we obtain
\begin{eqnarray*}
\SE_t(E-F)
&=&c_t(F^\vee)\SE_t(E)+c_t(F^\vee)c_{-t}(-E)\alpha_{2m}\sum_{i=1}^{2m-1}\gamma_{2m-i}^{(2m)}p_{2m-i}(F) t^{-i}.
\end{eqnarray*}
To finish the proof it suffices to observe that one has $\alpha_{2m}S_t(E)=\alpha_{2m}c_t(-E^\vee)$ since $\alpha_{2m}^2=0$.
\end{proof}

We conclude this section by giving a geometric interpretation to the Segre classes of virtual bundles. To achieve this, we first need to prove a lemma describing the top Chern class of the tensor product of a vector bundle and a line bundle.

\begin{lem} \label{lemma tensor}
Let $E$ be a vector bundle of rank $e$ and $L$ line bundle over $X\in \SM$. Denote by $x_i$ the Chern roots of $E$ and by $\tau$ that of $L$. Then in $I_{2m}^*(X)$ we have
\[
c_e(L\otimes E)=\sum_{l=0}^e c_l(E)\tau^{e-l}\cdot \left[1+\alpha_{2m} \sum_{j=1}^{2m-1}
\gamma_j^{(2m)} p_j(E)  \tau^{2m-j} \right].
\]
\end{lem}
\begin{proof} 
The proof essentially in a Chern root computation, together with the use of Lemma \ref{lemFGL2m}: 
\begin{align*}
c_e(L\otimes E)&= \prod_{i=1}^e (\tau\boxplus x_i)=\prod_{i=1}^e (\tau+x_i)\cdot  \prod_{i=1}^e\left[1+\alpha_{2m}\sum_{j=1}^{2m-1}\gamma_j^{(2m)}x_i^j \tau^{2m-j} \right].
\end{align*}
Indeed the right hand side is equal to the desired expression.
\end{proof}
\begin{prop}\label{push of tensor}
Let $E$ and $F$ two vector bundles over a smooth scheme $X$, respectively of rank $e$ and $f$. Consider the projective bundle of hyperplanes $\PP^*(E)\stackrel{\pi}\rightarrow X$, denote by $\mathcal{Q}$ the associated universal quotient line bundle $\calO(1)$ and by $\xi$ its first Chern class. Then for any nonnegative integer $s$ we have
\[
\pi_*\Big(\xi^s c_f(\calQ\otimes  F^{\vee})\Big)=\SE_{s+f-e+1}(E- F).
\]
\end{prop}
\begin{proof}
 In view of Lemma \ref{lemma tensor} and of the definition of (relative) Segre classes we have the following chain of equalities:
\begin{eqnarray*}
&&\pi_*\Big(\xi^s\cdot c_f(\calQ\otimes  F^{\vee})\Big)\\
&=&\pi_*\left( \sum_{l=0}^f c_l(F^{\vee})\left(\xi^{s+f-l}+ \alpha_{2m} \sum_{i=1}^{2m-1}\gamma_{2m-i}^{(2m)} p_{2m-i}(F^{\vee})  \xi^{s+i+f-l} \right)\right)\\
&=&\sum_{l=0}^f c_l(F^{\vee})\left(\SE_{(s+f-e+1)-l}(E)+\alpha_{2m} \sum_{i=1}^{2m-1}\gamma_{2m-i}^{(2m)} p_{2m-i}(F^{\vee}) \SE_{i+(s+f-e+1)-l}(E) \right)\\
&=&S_{s+f-e+1}(E-F).
\end{eqnarray*}
\end{proof} 
\section{Determinantal formula for Grassmannian bundles}
\subsection{Degeneracy loci and Kempf-Laksov resolutions}\label{secKL}
Let $E$ be a vector bundle of rank $n$ over a smooth quasi-projective variety $X$. Let $\Gr_d(E) \to X$ be the Grassmannian bundle of rank $d$ subbundles over $X$, \textit{i.e.}
\[
\Gr_d(E):=\{(x,S_x) \ |\ x \in X, S_x \mbox{ is a $d$-dimensional subspace of $E_x$} \}.
\]
Let $S$ be the tautological subbundle of the pullback of $E$ over $\Gr_d(E)$. Fix a complete flag $0=F^n\subset \cdots \subset F^1\subset F^0= E$ where the superscript indicates the corank, \textit{i.e.} $\rk\ F^k = n-k$. We denote also by $E$ and $F^i$ the pullback of  $E$ and $F^i$ along the projection $\Gr_d(E) \to X$ respectively. 

Let $\calP_d$ be the set of all partitions $(\lambda_1,\dots,\lambda_d)$ with at most $d$ parts. The length of $\lambda$ is the number of nonzero parts. Let $\calP_d(n)$ be the set of all partitions in $\calP_d$ such that $\lambda_i\leq n-d$ for all $i=1,\dots,d$. For each partition $\lambda \in \calP_d(n)$ of length $r$, consider the partial flag of $E$
\[
F_{\lambda}^{\bullet}: F^{\lambda_1-1+d} \subset F^{\lambda_2-2+d} \subset \cdots \subset F^{\lambda_r-r+d} \subset E.
\]
Define the type $A$ degeneracy locus $X_{\lambda}$ in $\Gr_d(E)$ by
\[
X_{\lambda}:=\{(x,S_x) \in \Gr_d(E) \ |\ \dim (F^{\lambda_i-i+d}_x \cap S_x) \geq i, i=1,\dots, r\}.
\]
Associated to the partial flag $F_{\lambda}^{\bullet}$ is a generalized flag bundle 
\[
\pi: \Fl(F_{\lambda}^{\bullet}) \to \Gr_d(E)
\]
such that the fiber at $p \in \Gr_d(E)$ consists of flags of subspaces $(D_1)_p \subset \cdots \subset (D_r)_p$ of $E_p$ with $\dim (D_i)_p=i$ and $(D_i)_p\subset F^{\lambda_i-i+d}_p$. Let $D_1\subset \cdots \subset D_r$ be the corresponding flag of tautological subbundles over $\Fl(F_{\lambda}^{\bullet})$. One can obtain the flag bundle $\Fl(F_{\lambda}^{\bullet})$  as a tower of projective bundles
\begin{eqnarray}\label{tower}
&&\pi: \Fl(F_{\lambda}^{\bullet})=\PP(F^{\lambda_r-r+d}/D_{r-1}) \stackrel{\pi_r}{\longrightarrow} 
\PP(F^{\lambda_{r-1}-(r-1)+d}/D_{r-2}) \stackrel{\pi_{r-1}}{\longrightarrow} \cdots \ \ \ \ \ \ \ \ \ \nonumber\\\label{P tower}
&&\ \ \ \ \ \ \ \ \ \ \ \ \ \ \ \ \ \ \ \ \ \ \cdots\stackrel{\pi_3}{\longrightarrow} \PP(F^{\lambda_2-2+d}/D_1) \stackrel{\pi_2}{\longrightarrow} \PP(F^{\lambda_1-1+d})  \stackrel{\pi_1}{\longrightarrow} \Gr_d(E).
\end{eqnarray}
We regard $D_i/D_{i-1}$ as the tautological line bundle  $\calO(-1)$ of $\PP(F^{\lambda_i-i+d}/D_{i-1})$. We denote the first Chern class $c_1((D_i/D_{i-1})^{\vee})$ by $\tau_i$.

\begin{defn}
Let $\widetilde{X}^{KL}_\lambda \subset \Fl(F_{\lambda}^{\bullet})$ be the locus where $D_r \subset S$.  It is well-known that $\widetilde{X}^{KL}_\lambda$ is smooth and birational to $X_{\lambda}$ along $\pi$ (see \cite{KempfLaksov}). We refer to it as the Kempf-Laksov resolution of the degeneracy loci $X_{\lambda}$. Define the Kempf-Laksov class $\kappa_{\lambda}$ in $I_{2m}^*(\Gr_d(E))$ by
\[
\kappa_{\lambda} := \pi_*[\widetilde{X}^{KL}_\lambda],
\]
where $[\widetilde{X}^{KL}_\lambda]$ is the fundamental class of $\widetilde{X}^{KL}_\lambda$ in $\Fl(F_{\lambda}^{\bullet})$.
\end{defn}
\begin{prop}\label{chern prod}
In $I_{2m}^*(\Fl(F_{\lambda}^{\bullet}))$, we have
\begin{equation}\label{Y_r prod}
\ [\widetilde{X}^{KL}_\lambda] =\prod_{i=1}^r c_{n-d}((D_i/D_{i-1})^{\vee}\otimes E/S).
\end{equation}
\end{prop}
This proposition follows from the following lemma and the same argument used in \cite[Section 4.3]{HIMN}.
\begin{lem}[Lemma 6.6.7 \cite{LevineMorel}, Example 14.1.1 \cite{FultonIntersection}]\label{gaussbonnet}
Let $V$ be a vector bundle of rank $e$ over $X$ and $s$ a section of $V$. Let $Z$ be the zero scheme of $s$. If $X$ is Cohen-Macaulay and the codimension of $Z$ is $e$, then $s$ is regular and 
\[
c_e(V)=[Z] \in I_{2m}^e(X).
\]
\end{lem}
In the next section, we will compute the class $\kappa_{\lambda}$ by pushing forward the product of Chern classes (\ref{Y_r prod}) along the tower of projective bundles (\ref{tower}). For that, we conclude this section with the following formula, which allows one to deal with each stage of the tower.
\begin{lem}[The $i$-th stage pushforward formula]\label{rth first} 
Let $\alpha_i:=c_{n-d}((D_i/D_{i-1})^{\vee}\otimes E/S)$ and  $\calA_m^{(\ell)}:=\calS_{m}((S- E/F^{\ell})^{\vee})$. We have
\begin{equation}\label{eq:push A}
\pi_{i*}\big(\tau_i^s\alpha_i\big)
=\sum_{p\geq 0}c_p(D_{i-1})\left(\calA_{\lambda_i+s-p}^{(\lambda_i-i+d)} + \alpha_{2m} \sum_{l=1}^{2m-1} \gamma_l^{(2m)} p_l(D_{i-1})\calA_{\lambda_i+s-p+2m-l}^{(\lambda_i-i+d)}\right)
\end{equation}
for all $s\geq 0$.
\end{lem}
\begin{proof}
We apply Proposition \ref{push of tensor} to $\PP^*((F^{\lambda_i-i+d}/D_{i-1})^{\vee})=\PP(F^{\lambda_i-i+d}/D_{i-1})$ with the tautological quotient bundle $\calQ=(D_i/D_{i-1})^{\vee}$. Then the result follows from Proposition \ref{relSegre}.
\end{proof}
\subsection{Computing $\kappa_{\lambda}$}\label{secAkappa}

First, we recall notations from \cite{HIMN} for the ring of certain formal Laurent series, necessary for the computation of the class $[X_{\lambda}]$. Let $R=\oplus_{m\in \ZZ}R_m$ be a commutative $\ZZ$-graded ring. Let  $t_1,\ldots,t_r$ be indeterminates. For ${\sfs}=(s_1,\ldots,s_r)\in \ZZ^r$, we denote $t^{\sfs}=t_1^{s_1}\cdots t_r^{s_r}$. A formal Laurent series of degree $m$ in the variables $t_1,\ldots,t_r$ with coefficients in $R$ is given by $f(t)=\sum_{\sfs\in \ZZ^r}a_{\sfs}t^{\sfs}$ 
where $a_{\sfs}\in R_{m-|{\sfs}|}$ for all ${\sfs}\in \ZZ^r$ with $|{\sfs}|=\sum_{i=1}^rs_i$. Its support $\supp f$ is defined as
\[
\supp f = \{ \sfs \in \ZZ^r \ |\ a_{\sfs} \not=0\}.
\]
For each $m\in \ZZ$, let $\scL_m^R$ denote the set of  all formal Laurent series $f(t)$ of degree $m$ such that there is $\sfm \in \ZZ^r$ such that $\sfm + \supp f$ is contained in the cone $C \subset \ZZ^r$ defined by inequalities $s_1\geq0, \; s_1+s_2\geq 0, \;\cdots, \; s_1+\cdots + s_{r} \geq 0$. The direct sum $\scL^R:=\oplus_{m\in \ZZ}\scL_m^R$ is a graded ring in an obvious manner. For each $i=1,\ldots,r,$ let $\scL^{R,i}$ denote the subring of $\scL^R$ consisting of series that do not contain negative power of $t_1,\ldots,t_{i-1}.$ In particular we have $\scL^{R,1}=\scL^R$. For each $m\in \ZZ$, let $R[[t_1,\ldots,t_r]]_m$ denote the set of formal power series in $t_1,\ldots,t_r$ of homogeneous degree $m$ with $\deg(t_i)=1.$  We define  the ring $R[[t_1,\ldots,t_r]]_{\gr}$ of graded formal power series to be $\oplus_{m\in \ZZ}R[[t_1,\ldots,t_r]]_m$. Note that $\scL^{R,i}$ is a graded $R[[t_1,\ldots,t_r]]_{\gr}$-module. 

Let us apply the above notation to $R=I_{2m}^*(\Gr_d(E))$, which is a graded algebra over $Q_{2m}$.  We can uniquely define a homomorphism of  graded $R[[t_1,\ldots,t_{i-1}]]_{\gr}$-modules 
\[
\phi_i:\scL^{R,i} \to I_{2m}^*(\mathbb{P}(F^{\lambda_{i-1}-i+1+d}/D_{i-2}))
\] 
by setting 
\[
t_1^{s_1}t_2^{s_2}\cdots t_r^{s_r}\mapsto  \tau_1^{s_1}\cdots \tau_{i-1}^{s_{i-1}}\calA_{s_{i}}^{(\lambda_i-i+d)}\cdots  \calA_{s_r}^{(\lambda_r-r+d)}.
\]
\begin{lem}\label{eq:push A phi}
Let $\alpha_i:=c_{n-d}((D_i/D_{i-1})^{\vee}\otimes E/S)$. We have
\begin{equation*}
\pi_{i*}\big(\tau_i^s\alpha_i\big)
=\phi_i\left(t_i^{\lambda_i+s}\left(\prod_{j=1}^{i-1} (1-t_j/t_i)\right)\left( 1+ \alpha_{2m} t_i^{2m}\sum_{l=1}^{2m-1} \gamma_l^{(2m)} p_l(-t_1/t_i,\dots,-t_{i-1}/t_i)\right)\right)
\end{equation*}
for all $s\geq 0$.
\end{lem}
\begin{proof}
By (\ref{eq:push A}), we find
\begin{eqnarray*}
&&\pi_{i*}\big(\tau_i^s\alpha_i\big)\\
&=&\phi_i\left(t_i^{\lambda_i+s}\sum_{p\geq 0}e_p(-t_1,\dots,-t_{i-1})t_i^{-p}\left( 1+ \alpha_{2m} t_i^{2m}\sum_{l=1}^{2m-1} \gamma_l^{(2m)} p_l(-t_1,\dots,-t_{i-1})t_i^{-l} \right)\right).
\end{eqnarray*}
Thus the claim follows from the definitions of $e_p$ and $p_l$.
\end{proof}
Recall that the multi-Schur determinant is denoted as follows. Let $c^{(i)}=\{ c_k^{(i)}, k\in \ZZ\}$ be a set of infinite variables for each $i=1,\dots,r$. Let $\ell=(\ell_1,\dots, \ell_r) \in \ZZ^r$.
\[
s_{\ell}(c^{(1)},\dots, c^{(r)}) := \det(c_{\ell_i+j-i}^{(i)})_{1\leq i,j\leq r}.
\]
In this notation, we have
\[
\Det[t_1^{\ell_1}\cdots t_r^{\ell_r}] = s_{\ell}(c^{(1)},\dots, c^{(r)})|_{c_k^{(i)}=t_i^k}.
\]
and
\[
\Det[\calA^{(k_1)}_{\ell_1}\cdots \calA^{(k_r)}_{\ell_r}]=s_{\ell}(c^{(1)},\dots, c^{(r)})|_{c_k^{(i)}=\calA_k^{(i)}}.
\]
\begin{thm}\label{mainthmA}
Let $\lambda$ be a partition in $\calP_d(n)$ with length $r$. Let $\calA_m^{(\ell)}:=\calS_{m}((S- E/F^{\ell})^{\vee})$ and $k_i:=\lambda_i-i+d$. The Kempf-Laksov class $\kappa_{\lambda}$ in $I_{2m}^*(\Gr_d(E))$ is given by
\begin{eqnarray*}
\kappa_{\lambda} &=& \Det\big[\calA^{(k_1)}_{\lambda_1}\cdots \calA^{(k_r)}_{\lambda_r}\big]\\
&&\ \ \ \ \ +
\alpha_{2m} \sum_{l=-m+1}^{m-1} (-1)^{m+l}\gamma_{m+l}^{(2m)} \left( \sum_{1\leq a<b \leq r} 
\Det\big[\calA^{(k_1)}_{\lambda_1}\cdots \calA^{(k_a)}_{\lambda_a+m+l}\cdots \calA^{(k_b)}_{\lambda_b+m-l}  \cdots \calA^{(k_r)}_{\lambda_r}\big]
\right).
\end{eqnarray*}
\end{thm}
\begin{proof}
It follows from the repeated application of Lemma \ref{eq:push A phi} to (\ref{chern prod}), and the following identity of Vandermonde determinant:
\begin{eqnarray*}
&&\prod_{i=1}^r t_i^{\lambda_i}\left(\prod_{1\leq i<j\leq r} \left(1-t_it_j^{-1}\right)\right)
\left( 1+\alpha_{2m} \sum_{l=-m+1}^{m-1} (-1)^{m+l}\gamma_{m+l}^{(2m)} \left( \sum_{1\leq i<j\leq r} t_i^{m+l}t_j^{m-l} \right)\right)\\
&=&
\Det\big[t_1^{\lambda_1}\cdots t_r^{\lambda_r}\big]
+
\alpha_{2m} \sum_{l=-m+1}^{m-1} (-1)^{m+l}\gamma_{m+l}^{(2m)} \left( \sum_{1\leq a<b \leq r} 
\Det\big[t_1^{\lambda_1}\cdots t_a^{\lambda_a+m+l}\cdots t_b^{\lambda_b+m-l}  \cdots t_r^{\lambda_r}\big]
\right).
\end{eqnarray*}
\end{proof}
{
\section{Pfaffian formula for Lagrangian Grassmannian}\label{secLagC}
\subsection{Lagrangian degeneracy loci and its Kempf-Laksov resolutions}
Let $E$ be a symplectic vector bundle over $X$ with rank $2n$, \textit{i.e.}, we are given a nowhere vanishing section of $\wedge^2 E$. For a subbundle $F$ of $E$, we denote by $F^{\perp}$  the orthogonal complement of $F$ with respect to the symplectic form.
Fix a complete flag $F^{\bullet}$ of subbundles of $E$
\[
0=F^n\subset F^{n-1} \subset \cdots \subset F^1 \subset F^0 \subset F^{-1} \subset \cdots \subset F^{-n}=E,  
\]
such that  $\rk F^i=n-i$ and $(F^i)^{\perp} = F^{-i}$ for all $i$. The condition implies that $F^i$ is isotropic for each $i\geq0$. In particular, $F^0$ is Lagrangian. Let $\LG(E) \to X$ be the Lagrangian Grassmannian bundle over $X$, \textit{i.e.}, the fiber at $x \in X$ is the Grassmannian $\LG(E_x)$ of the Lagrangian subspaces of $E_x$. As before, for vector bundles we will omit the pullbacks from the notation. Let $L$ be the tautological vector bundle over $\LG(E)$. 

A strict partition $\lambda$ is an infinite sequence $(\lambda_1,\lambda_2,\cdots)$ of nonnegative integers such that all but finitely many $\lambda_i$'s are zero, and such that $\lambda_i>0$ implies $\lambda_i>\lambda_{i+1}$. The length of $\lambda$ is the number of nonzero parts. Let $\SP(n)$ be the set of all strict partitions with length at most $n$ and $\lambda_1\leq n$.

For each  strict partition $\lambda$ in $\SP(n)$ of length $r$, consider the following partial flag $F_{\lambda}^{\bullet}$ of $F^{\bullet}$
\[
F_{\lambda}^{\bullet}: \ \ F^{\lambda_1-1} \subset \cdots \subset F^{\lambda_r-1} \subset F^0.
\]
The corresponding Lagrangian degeneracy loci $X_{\lambda} \subset \LG(E)$ is given by
\[
X_{\lambda}^C = \{ (L_x, x)\in \LG(E) \ |\ \dim ( F^{\lambda_i-1}_x \cap L_x) \geq i, \ \ i=1,\dots, r\}.
\]
Consider the generalized flag bundle $\pi: \Fl(F_{\lambda}^{\bullet}) \to \LG(E)$ associated to the partial flag $F_{\lambda}^{\bullet}$ of $E$. As before, the flag bundle $\Fl(F_{\lambda}^{\bullet})$ can be obtained as a tower of projective bundles
\begin{eqnarray*}
&&\pi: \Fl(F_{\lambda}^{\bullet})=\PP(F^{\lambda_r-1}/D_{r-1}) \stackrel{\pi_r}{\longrightarrow} \PP(F^{\lambda_{r-1}-1}/D_{r-2}) \stackrel{\pi_{r-1}}{\longrightarrow} \cdots \ \ \ \ \ \ \ \ \ \ \ \ \ \ \ \ \ \ \ \ \ \ \ \ \\
&&\ \ \ \ \ \ \ \ \ \ \ \ \ \ \ \ \ \ \ \ \ \ \ \ \ \ \ \ \ \ \ \ \ \ \ \ \ \ \ \ \ \ \ \ \ \ \ \  \stackrel{\pi_3}{\longrightarrow} \PP(F^{\lambda_2-1}/D_1) \stackrel{\pi_2}{\longrightarrow} \PP(F^{\lambda_1-1})  \stackrel{\pi_1}{\longrightarrow} \LG(E).
\end{eqnarray*}
We regard $D_i/D_{i-1}$ as the tautological line bundle of $\PP(F^{\lambda_i-1}/D_{i-1})$ and denote $\tau_i:=c_1((D_i/D_{i-1})^{\vee})$.
\begin{defn}\label{defLagKL}
Let $\widetilde{\frakX}^{KL}_\lambda \subset \Fl(F_{\lambda}^{\bullet})$ be the locus where $D_r \subset L$.  It is known that $\widetilde{\frakX}^{KL}_\lambda$ is smooth and birational to $X^C_{\lambda}$ along $\pi$ (\textit{cf.} \cite{HIMN}). We call it the Lagrangian Kempf-Laksov resolution of the degeneracy loci $X^C_{\lambda}$. Define the Lagrangian Kempf-Laksov class $\kappa^C_{\lambda} \in I_{2m}^*(\LG(E))$ associated to a strict partition $\lambda \in \SP(n)$ by
\[
\kappa^C_{\lambda} :=[\widetilde{\frakX}^{KL}_\lambda \to \LG(E)] =\pi_*[\widetilde{\frakX}^{KL}_\lambda \to \Fl(F_{\lambda}^{\bullet})].
\]
\end{defn}
By adapting \cite[Section 6.3]{HIMN} to the current situation we obtain the following expression.
\begin{prop}\label{propLagKL}
In $I_{2m}^*(\Fl(F_{\lambda}^{\bullet}))$, we have
\begin{eqnarray}\label{Yr prod}
[\widetilde{\frakX}^{KL}_\lambda \to \Fl(F_{\lambda}^{\bullet})] &=&\prod_{i=1}^r c_{n-i+1}((D_i/D_{i-1})^{\vee}\otimes D_{i-1}^{\perp}/L),
\end{eqnarray}
and as a consequence, we have
\begin{equation}\label{kappaCprod}
\kappa^C_{\lambda} = \pi_*\left(\prod_{i=1}^r c_{n-i+1}((D_i/D_{i-1})^{\vee}\otimes D_{i-1}^{\perp}/L)\right).
\end{equation}
\end{prop}
To compute $\kappa^C_{\lambda}$ by pushing forward (\ref{Yr prod}) along the projective tower, we conclude this section with the following formula, which allows one to deal with each stage of the tower.
\begin{lem}\label{lem: kazarian push C}
Let $\calC_k^{(\ell)}:=\calS_{k}((L-E/F^{\ell})^{\vee})$ for each $k\in \ZZ$ and $\ell \in \{0,1,\dots, n-1\}$and $\alpha_i:=c_{n-i+1}((D_i/D_{i-1})^{\vee}\otimes D_{i-1}^{\perp}/L)$. For each integer $s\geq 0$, we have
\begin{eqnarray*}
\pi_{i*}\left(\tau_i^s\alpha_i\right)
&=&\sum_{q=0}^{\infty}
c_q(D_{i-1} - D_{i-1}^{\vee})
\left(\calC_{\lambda_i+s-q}^{(\lambda_i-1)}+ \alpha_{2m} \sum_{a=1}^{2m-1} \gamma_a^{(2m)} p_a(D_{i-1} - D_{i-1}^{\vee})\calC_{\lambda_i+s-q+2m-a}^{(\lambda_i-1)}\right).
\end{eqnarray*}
\end{lem}
\begin{proof}
We apply Proposition \ref{push of tensor}:
\[
\pi_{i*}\left(\tau_i^s\alpha_i\right)=\calS_{\lambda_i+s}((F^{\lambda_i-1}/D_{i-1})^{\vee} - (D_{i-1}^{\perp}/L)^{\vee})=\calS_{\lambda_i+s}((L-E/F^{\lambda_i-1})^{\vee}-(D_{i-1} - D_{i-1}^{\vee})^{\vee}).
\]
Here we have used the identity $D_{i-1}^{\perp} = E/D_{i-1}^{\vee}$. Now the claim follows from Proposition \ref{relSegre}.
\end{proof}
\subsection{Computing the class $\kappa^C_{\lambda}$}
Let $R=I_{2m}^*(\LG(E))$ and set $\calC_k^{(\ell)}:=\calS_{k}((L-E/F^{\ell})^{\vee})$ as before. Adapting the notations from Section \ref{secAkappa}, we define a homomorphism of graded $R[[t_1,\ldots,t_{i-1}]]_{\gr}$-modules 
\[
\phi_i^C:\scL^{R,i} \to I_{2m}^*(\mathbb{P}(F^{\lambda_{i-1}-1}/D_{i-2}))
\] 
by setting 
\[
t_{1}^{s_{1}}\cdots t_r^{s_r}\mapsto  \tau_1^{s_1}\cdots \tau_{i-1}^{s_{i-1}}\calC_{s_{i}}^{(\lambda_i-1)}\cdots  \calC_{s_r}^{(\lambda_r-1)}.
\]
\begin{lem}\label{lempiphiC}
We have
\begin{eqnarray*}
\pi_{i*}\left(\tau_i^s\alpha_i\right)
&=&\phi_i\left(t_i^{\lambda_i+s}\left(\prod_{j=1}^{i-1}\frac{1-t_j/t_i}{1+t_j/t_i}\right)\left(1-2\alpha_{2m} \sum_{q=1}^{m} \gamma_{2q-1}^{(2m)} \sum_{1\leq j < i}  t_j^{2q-1}t_i^{2m-2q+1}\right)\right).
\end{eqnarray*}
\end{lem}
\begin{proof}
For convenience, define the polynomial $H_p(t_1,\dots, t_j)$ by the generating function
\[
\sum_{p=0}^{\infty} H_p(t_1,\dots, t_j) u^p = \prod_{k=1}^j \frac{1- t_k u}{1+ t_k u}.
\] 
Note that one has $c_p(D_{i-1} - D_{i-1}^{\vee}) = H_p(\tau_1,\dots,\tau_{i-1})$. From Lemma \ref{lemma Power}, we observe that
\begin{eqnarray*}
\sum_{l=1}^{2m-1} \gamma_l^{(2m)} p_l(D_{i-1} - D_{i-1}^{\vee})\calC_{\lambda_i+s-p+2m-l}^{(\lambda_i-1)}
&=&\sum_{l=1}^{2m-1} \gamma_l^{(2m)} p_l(\tau_1,\dots,\tau_{i-1})((-1)^l -1)\scC_{\lambda_i+s-p+2m-l}^{(\lambda_i-1)}\\
&=&-2\sum_{q=1}^{m} \gamma_{2q-1}^{(2m)} p_{2q-1}(\tau_1,\dots,\tau_{i-1})\scC_{\lambda_i+s-p+2m-2q+1}^{(\lambda_i-1)}.
\end{eqnarray*}
Therefore, by Lemma \ref{lem: kazarian push C} and the definition of $\phi_i$, it follows that
\begin{eqnarray*}
&&\pi_{i*}\left(\tau_i^s\alpha_i\right)\\
&=&\phi_i\left(\sum_{p\geq 0}H_p(t_1,\dots,t_{i-1})\left(t_i^{\lambda_i+s-p}-2\alpha_{2m} \sum_{q=1}^{m} \gamma_{2q-1}^{(2m)} p_{2q-1}(t_1,\dots,t_{i-1})t_i^{\lambda_i+s-p+2m-2q+1}\right)\right)\\
&=&\phi_i\left(t_i^{\lambda_i+s}\sum_{p\geq 0}H_p(t_1,\dots,t_{i-1})t_i^{-p}\left(1-2\alpha_{2m} \sum_{q=1}^{m} \gamma_{2q-1}^{(2m)} p_{2q-1}(t_1,\dots,t_{i-1})t_i^{2m-2q+1}\right)\right).
\end{eqnarray*}
Now the claim is a consequence of the definitions of $H_p$ and $p_{\ell}$.
\end{proof}
\begin{thm}
Let $\lambda$ be a strict partition in $\SP(n)$ of length $r$. The corresponding Lagrangian Kempf-Laksov class $\kappa_{\lambda}^C$ is given by
\begin{eqnarray*}
\kappa_{\lambda}^C
&=&\Pf\big[ \scC^{(\lambda_1-1)}_{\lambda_1}\cdots \scC^{(\lambda_r-1)}_{\lambda_r}  \big] \\
&&\ \ \ - 2\alpha_{2m}\sum_{q=1}^{m}\gamma_{2q-1}^{(2m)} \sum_{1\leq i<j\leq r}  \Pf\big[ \scC^{(\lambda_1-1)}_{\lambda_1}\cdots \scC^{(\lambda_i-1)}_{\lambda_i+2q-1}\cdots \scC^{(\lambda_j-1)}_{\lambda_j+2m-2q+1}\cdots \scC^{(\lambda_r-1)}_{\lambda_r}  \big].
\end{eqnarray*}
where $\Pf\big[ \scC^{(\lambda_1-1)}_{\ell_1}\cdots \scC^{(\lambda_r-1)}_{\ell_r}  \big]$ are the multi-Schur Pfaffian used by Kazarian \cite{Kazarian}.
\end{thm}
\begin{proof}
The repeated application of Lemma \ref{lempiphiC} to (\ref{kappaCprod}) yields
\[
\kappa_{\lambda} = \phi_1\left(\left(\prod_{i=1}^rt_i^{\lambda_i} \right)\left(\prod_{1\leq i < j\leq r} \frac{1-t_i/t_j}{1+t_i/t_j}\right) \left(1 - 2\alpha_{2m}\sum_{q=1}^{m}\gamma_{2q-1}^{(2m)} \sum_{1\leq i<j\leq r}  t_i^{2q-1}t_{j}^{2m-2q+1} \right)\right).
\]
Now the claim follows by noticing that we can write the function of $t$'s in the right hand side as
\[
\Pf\big[ t_1^{\lambda_1}\cdots t_r^{\lambda_r}  \big] - 2\alpha_{2m}\sum_{q=1}^{m}\gamma_{2q-1}^{(2m)} \sum_{1\leq i<j\leq r}  \Pf\big[ t_1^{\lambda_1}\cdots t_i^{\lambda_i+2q-1}\cdots t_j^{\lambda_j+2m-2q+1}\cdots t_r^{\lambda_r}  \big].
\]
\end{proof}
}
\section{Appendix: Symmetric functions and their evaluation on virtual vector bundles}
Although there exist several families that can be used as a basis of the algebra $\Lambda$ of symmetric functions, for this purpose we will make use of the \textit{elementary} symmetric functions $\{e_i\}_{i\in \NN}$. Besides this family, in our treatment we will also need the \textit{complete} symmetric functions $\{h_i\}_{i\in \NN}$ and the \textit{power sum} symmetric functions $\{p_i\}_{i\in \NN}$. We follow the usual convention of setting $e_0=h_0=p_0=1$ and sometimes we also allow indices to be negative, in which cases the associated symmetric function is understood to be zero.
 
Let us first recall some known formulas relating these three families and an alternative description of complete symmetric functions
\begin{lem}\label{lem power} 
For any  $k\in\NN\setminus\{0\}$, the following equalities hold:
\begin{align*}
&a)\ p_k=(-1)^{k+1} k\cdot e_k-\sum_{i=1}^{k-1}(-1)^{i} p_{k-i}\cdot e_i\ ;\\
&b)\  k\cdot e_k=\sum_{i=1}^k (-1)^{i+1}p_i\cdot e_{k-i}=\sum_{i=0}^{k-1} (-1)^{k+1-i} p_{k-i}\cdot e_i\ ;\\
&c)\  p_k=k\cdot h_k-\sum_{i=1}^{k-1} p_i \cdot h_{k-i}\ ;\\
&d)\   k\cdot h_k=\sum_{i=1}^k p_i\cdot h_{k-i}=\sum_{i=0}^{k-1} p_{k-i}\cdot h_i.
\end{align*}
\end{lem}
\begin{proof}
See \cite[Chapter 6, exercise 1]{YoungFulton}.
\end{proof}
Our interest in $\Lambda$ is due to the fact that its elements can be evaluated on vector bundles. If we consider an oriented cohomology theory $A^*$, then to any vector bundle $E\rightarrow X$ we can associate a ring homomorphism 
$$\varphi_E^A:\Lambda\rightarrow A^*(X)$$
 which maps $e_i$ to $c_i(E)$. Since elementary symmetric functions form a multiplicative basis of $\Lambda$, the morphism is defined on the whole domain. In general one can define a ring homomorphism $\varphi^A_\alpha$ for any element of $\alpha\in K^0(X)$, but we will be only interested in the classes of the form $[E]-[F]$, in which case one sets $\varphi^A_{[E]-[F]}(e_i)=c_i(E-F)$. From now on we will fix $A^*=I_{2m}^*$ and, in order to simplify the notation, we will write $f(E-F)$ instead of $\varphi^{I_{2m}}_{[E]-[F]}(f)$. 

As the following lemma clarifies, in the special case of the power sum symmetric functions,  the evaluation on a virtual bundle $[E]-[F]$ can be easily expressed by means of the evaluation on the two summands.   
\begin{lem} \label{lemma Power}
For bundles $E$ and $F$ and any strictly positive integer $k$, we have 
\[
p_k(E-F)=p_k(E)-p_k(F).
\]
\end{lem}
\begin{proof}
The proof is by induction on $k$, with the base of the induction being $k=1$. Since $p_1=e_1=h_1$, we have
\[
p_1(E-F)=c_1(E-F)=c_1(E)-h_1(F)=p_1(E)-p_1(F).
\]
In view of Lemma \ref{lem power}, $a)$, and of the inductive hypothesis, we have
\begin{eqnarray*}
&&p_k(E-F)\\
&=&(-1)^{k+1}k\cdot c_{k}(E-F)-\sum_{i=1}^{k-1}(-1)^{i}p_{k-i}(E-F)c_i(E-F)\\
&=&(-1)^{k+1}k\sum_{j=0}^k (-1)^j c_{k-j}(E)\cdot h_j(F)
-\sum_{i=1}^{k-1}\left[(-1)^{i} \Big[p_{k-i}(E)-p_{k-i}(F)\Big]\sum_{j=0}^i (-1)^j c_{i-j}(E)\cdot h_j(F) \right]\\
\end{eqnarray*}
Now we want to separate the summands into three different families: the terms that only involve $E$, those that only involve $F$, and the mixed terms. For this we rewrite the last expression.
\begin{eqnarray*}
&&p_k(E-F)\\
&=&\left[(-1)^{k+1}k\cdot c_k(E)+\sum_{i=1}^{k-1}(-1)^{i+1} p_{k-i}(E)\cdot c_i(E)\right]+\left[ -k\cdot h_k(F)+ \sum_{i=1}^{k-1} p_{k-i}(F) \cdot h_i(F)\right]\\
&&+(-1)^{k+1}k \sum_{j=1}^{k-1} (-1)^j c_{k-j}(E) \cdot h_j(F)\\
&&-\sum_{i=1}^{k-1}(-1)^{i} p_{k-i}(E)\sum_{j=1}^i (-1)^j c_{i-j}(E)h_j(F)+\sum_{i=1}^{k-1}(-1)^{i} p_{k-i}(F)\sum_{j=0}^{i-1} (-1)^j c_{i-j}(E)h_j(F).
\end{eqnarray*}
In view of Lemma \ref{lem power}, $a)$ and $c)$, the first line on the right hand side is $p_k(E)-p_k(F)$. We now want to show that the last three terms add up to 0. For this we split the middle term into two parts, which  we combine with the rest.  
\begin{eqnarray*}
&&p_k(E-F)-p_k(E)+p_k(F)\\
&=&\sum_{j=1}^{k-1}(-1)^{k+1-j}(k-j)\cdot h_j(F) \cdot c_{k-j}(E)+\sum_{j=1}^{k-1}(-1)^{k+1-j} j\cdot h_j(F) \cdot c_{k-j}(E)\\
&&-\sum_{i=1}^{k-1}\left[p_{k-i}(E)\sum_{j=1}^i (-1)^{j-i} c_{i-j}(E)\cdot h_j(F)\right]+\sum_{i=1}^{k-1}\left[ p_{k-i}(F)\sum_{j=0}^{i-1} (-1)^{j-i} c_{i-j}(E) \cdot h_j(F)\right]\\
&=&\sum_{j=1}^{k-1} h_j(F)\left[(-1)^{k-j+1}(k-j) \cdot c_{k-j}(E)-\sum_{i=j}^{k-1}(-1)^{j-i}p_{k-i}(E) \cdot c_{i-j}(E)\right]\\
&&+ \sum_{l=1}^{k-1}(-1)^l c_l(E)\left[ -(k-l)\cdot h_{k-l}(F)+\sum_{j=0}^{(k-l)-1} p_{(k-l)-j}(F)\cdot h_{j}(F)\right].
\end{eqnarray*}
It now suffices to apply Lemma \ref{lem power},\,$b)$ and $d)$ to complete the proof.
\end{proof}

\bibliography{references}{}
\bibliographystyle{acm}



\end{document}